\documentclass[paperpaper,10pt,conference]{ieeeconf}

\IEEEoverridecommandlockouts
\let\labelindent\relax

\usepackage{cite}

\usepackage{amsmath,amssymb,amsthm}
\usepackage{graphicx}
\usepackage{epsfig}
\usepackage{xcolor}
\usepackage{enumitem}
\usepackage{float}

\usepackage{hyperref}
\hypersetup{
colorlinks=true,
linkcolor=blue,
filecolor=blue,
urlcolor=blue,
citecolor=blue
}
\usepackage[linesnumbered,ruled]{algorithm2e}
\usepackage{caption}
\usepackage{subcaption}
\usepackage{mathptmx}
\usepackage{tikz}
\usetikzlibrary{arrows,shapes,calc}
\usetikzlibrary{calc,intersections,through,backgrounds}
\usetikzlibrary{positioning}

\usepackage{pgfplots}
\pgfplotsset{compat=1.16}
\usepackage{pgfplotstable}

\usepackage{textgreek}
\usepackage{textcomp}

\definecolor{myblue}{RGB}{0,114,178}
\definecolor{mygreen}{RGB}{0,158,115}

\theoremstyle{theorem}
\newtheorem{theorem}{Theorem}
\newtheorem{lemma}{Lemma}
\newtheorem{corollary}{Corollary}

\theoremstyle{definition}
\newtheorem{definition}{Definition}

\newtheorem{remark}{Remark}
\newtheorem{example}{Example}
\newtheorem{problem}{Problem}

\newcommand{\tb}[1]{\textcolor{black}{#1}}

\def\BibTeX{{\rm B\kern-.05em{\sc i\kern-.025em b}\kern-.08em
T\kern-.1667em\lower.7ex\hbox{E}\kern-.125emX}}

\begin{document}

\title{\LARGE \bf
Guaranteed Cost Structured Control in Infinite-Horizon Linear-Quadratic Cooperative Differential Games
}
\author{
Aniruddha Roy and Pavankumar Tallapragada, \IEEEmembership{Member, IEEE}
\thanks{A. Roy and P. Tallapragada are with the Robert Bosch Centre for Cyber-Physical Systems, Indian Institute of Science, Bengaluru 560012, Karnataka, India. A. Roy's research is supported by the Anusandhan National Research Foundation (ANRF) through the National Post-Doctoral Fellowship (N-PDF) under Grant PDF/2025/003882 (e-mail: \{aniruddharoy, pavant\}@iisc.ac.in).}
\thanks{This work has been accepted at the 2026 IEEE 65th Conference on Decision and Control (CDC 2026), to be held in Honolulu, Hawaii, USA from December 15–18, 2026. The final published version will be copyrighted by IEEE.}
}
\maketitle
\thispagestyle{empty}
\pagestyle{empty}

\begin{abstract}	
\tb{In this paper, we consider the infinite-horizon 
linear-quadratic cooperative differential games 
with output feedback information structure. We 
first show that computing Pareto optimal 
controls under output feedback is difficult 
even for low-dimensional games. To address this, 
we introduce the concept of feedback guaranteed 
cost structured control (GCSC). At a feedback GCSC, the total weighted team cost remains below a prescribed threshold while satisfying the structural constraint. We derive monotonicity properties of the feedback 
GCSC set and the admissible weight set, respectively. Further, we show that Pareto optimal controls (if they exist) belong to the class of feedback GCSCs. We provide performance measures of the Pareto optimal 
controls and the proposed GCSC relative to the 
output feedback optimal control. We also establish verification and synthesis conditions for a feedback GCSC using linear 
matrix inequalities, where the synthesis formulation is convex and requires no 
semi-definite programming relaxation. Finally, we illustrate the effectiveness of the proposed approach through numerical examples, including a microgrid 
tracking synchronization case study.} 
\end{abstract}
\section{Introduction}
\tb{Differential game theory provides a mathematical framework for analyzing strategic interactions among multiple players \cite{Basar:99}. A cooperative differential game (CDG) arises when players coordinate their control actions to optimize a collective objective \cite{ Engwerda:08CDG}. Pareto optimality has been widely studied in CDGs, with applications in areas such as power systems \cite{Chen:15, Mu:2019}, human-robot coordination \cite{An:2023cooperative}, and consensus problems \cite{Kazerooni:09}. Pareto optimality under an open-loop information structure in a finite-horizon setting was studied in \cite{Engwerda:2010SIAM} and later extended to the infinite-horizon case in \cite{Reddy:2013Automatica, Reddy:2014TAC}. In particular, \cite{Reddy:2013Automatica} established necessary and sufficient conditions for Pareto optimality in infinite-horizon linear-quadratic cooperative differential games (LQ-CDGs) and showed its connection to a weighted-sum optimal control problem. }

\tb{Despite these advances, many real-world systems consist of multiple interacting agents, commonly referred to as multi-agent systems (MAS). In such systems, agents are often heterogeneous and communicate over a directed network, where each agent typically has access only to local (or output) information. Work in \cite{Kazerooni:09} primarily addressed consensus in homogeneous MAS with decoupled dynamics under an undirected graph. Moreover, \cite{Reddy:2013Automatica, Reddy:2014TAC} studied the Pareto optimality problem under an open-loop information structure. However, the problem of obtaining Pareto optimal solutions in a CDG setting under output feedback information remains largely unexplored. The control of heterogeneous MAS using the guaranteed cost equilibrium concept was studied for networked differential games in a non-cooperative setting in \cite{Roy:22}, and later extended to a general class of linear-quadratic non-cooperative differential games (LQ-NCDGs) with output feedback information in \cite{Roy:2025DGAA}. We stress that the focus of the present work is fundamentally different from \cite{Roy:2025DGAA}. Specifically, \cite{Roy:2025DGAA} studies an equilibrium concept in the LQ-NCDG setting, whereas the current work addresses a cooperative control problem in a CDG framework under output feedback information.} 

\tb{In this work, we propose a feedback guaranteed cost structured control (GCSC) concept for infinite-horizon LQ-CDGs with output feedback information structure. Note that, we consider a general class of LQ-CDGs with non-transferable utility (NTU), where all players form the grand coalition. The formation of sub-coalitions and utility transfers (side payments) during the game are not considered; see \cite{Petrosyan:06, Basar:18} for related discussions. Inspired from suboptimal control \cite{Iwasaki:94, Leibfritz:01, Jiao:2019TAC}, the objective is to design a structured cooperative controller that ensures the weighted team cost remains below a prescribed bound while satisfying the structural constraint. To the best of our knowledge, the GCSC problem for LQ-CDGs with NTU under an output feedback information structure has not been studied in the existing literature. The main contributions of this paper are: (i) we introduce 
the notion of feedback GCSC for a general class of 
LQ-CDG (ii) we derive fundamental 
properties of the feedback GCSC set and the admissible 
weight set, including monotonicity 
(Lemma~\ref{lem:monotonicity_properties}) and the 
inclusion of Pareto optimal controls within the 
GCSC set (Theorem~\ref{thm:pareto_gc}), (iii) we 
quantify the associated performance measures (Lemma~\ref{lem:eta1_eta2_bound}), 
and (iv) we derive verification 
(Theorem~\ref{thm:GC_verification}) and synthesis (Theorem~\ref{thm:GC_cop_synthesis}) conditions for a feedback GCSC. We emphasize that the synthesis formulation 
gives a convex feasibility problem. As a 
result, semi-definite programming (SDP) 
relaxation is not required, in contrast to 
\cite{Roy:2025DGAA}; see 
Remark~\ref{rem:comparision_Roy2025} for a 
detailed comparison.
}

The remainder of this paper is organized as follows. Section~\ref{sec:Prelim} presents the 
preliminaries and problem formulation. 
Section~\ref{sec:OFGCC} introduces the feedback 
GCSC and its fundamental properties. 
Section~\ref{sec:verfication_existence} provides 
verification and synthesis conditions. 
Section~\ref{sec:numerical} presents numerical 
examples, and Section~\ref{sec:concl} concludes.

\subsubsection*{Notation} 
\tb{$\mathbb{R}^{n\times m}$ denotes real matrices 
of size $n\times m$; $\mathbb{R}^n := 
\mathbb{R}^{n\times 1}$. $E^\prime$ denotes 
transpose. ${I}_n$ and ${0}_{n\times m}$ denote 
the identity and zero matrices. For 
$x \in \mathbb{R}^n$, $\|x\| := \sqrt{x^\prime x}$. An $n$-tuple
$(V_1,\ldots,V_n)$ is written as $(V_i,V_{-i})$ 
where $V_{-i}$ collects all elements except $V_i$. 
$M\succ 0$ ($M\succeq 0$) denotes positive 
definite (semi-definite). $W\otimes X$ and 
$W\oplus X$ denote the Kronecker product and 
direct sum. $A\in\mathbb{R}^{n\times n}$ is 
Hurwitz if all eigenvalues have negative real 
parts. $\mathcal{B}(r):=\{z\in\mathbb{R}^n\mid 
\|z\|\le r\}$, $r>0$ (given). Let 
$\mathsf{G}=(\mathsf{N},\mathsf{E})$ be a 
directed graph with node set 
$\mathsf{N}=\{1,\dots,N\}$ and edge set 
$\mathsf{E}\subseteq\mathsf{N}\times\mathsf{N}$, where $(j,i)\in\mathsf{E}$ denotes a link from 
$j$ to $i$ and $\mathsf{N}_i:=\{j\in\mathsf{N} 
\mid(j,i)\in\mathsf{E}\}$ is the in-neighbor 
set of node $i$.}
\section{Preliminaries and problem formulation}
\label{sec:Prelim}
In this section, we first introduce the required preliminaries on Pareto optimality in infinite-horizon LQ-CDG. We then present the problem.  
\subsection{Preliminaries}
Let $\mathsf{N} := \{1,2,\ldots,N\}$ denote the set of players. We assume that the state vector $x(t)\in\mathbb{R}^n$ evolves according to the following linear time-invariant (LTI) dynamics 
\begin{subequations}
\label{eq:LQDG}
\begin{align}
\dot{x}(t) = Ax(t)+\sum_{i\in\mathsf{N}} B_i u_i(t), ~ x(0)=x_0,
\label{eq:state_dynamics}
\end{align}
where player $i\in \mathsf{N}$ applies a control input $u_i(t)\in\mathbb{R}^{m_i}$ at each time instant $t \in[0,\infty)$. Here,  $A\in\mathbb{R}^{n\times n}$, $B_i\in\mathbb{R}^{n\times m_i}$, and $x_0\in\mathbb{R}^n$ is the initial condition. We assume that each player $i$ has only partial state information. In particular, player $i \in \mathsf{N}$ has access to the partial state or output $y_i(t)\in\mathbb R^{s_i}$ given by
\begin{align}
y_i(t) = C_i x(t),
\label{eq:output_vector}
\end{align}
where $C_i\in\mathbb R^{s_i\times n}$, $\mathrm{rank}(C_i) = s_i$, and $s_i\le n$. 
The objective of player $i\in\mathsf N$ is to determine $u_i(\cdot)$ that minimizes the following cost functional 
 \begin{align}
 J_i(x_0, u_i,u_{-i}) := \int_0^{\infty} \left( y_i^\prime(t)Q_i y_i(t) + u_i^\prime(t)R_i u_i(t)\right)dt,
 \label{eq:objective_IFH_output}
 \end{align}
 where $Q_i\in\mathbb R^{s_i\times s_i}$, $Q_i\succeq 0$, and $R_i\in\mathbb R^{m_i\times m_i}$, $R_i\succ 0$. Since each player has access only to output information \eqref{eq:output_vector}, we assume an output feedback information structure. In particular, each player uses a static output feedback control of the following form \eqref{eq:fbcontrol} to achieve their objectives. 
\begin{align} 
		u_i(t)=F_i y_i(t),~F_i\in \mathbb R^{m_i\times s_i}. \label{eq:fbcontrol}
	\end{align} 
\end{subequations} 
Note that player $i$'s cost \eqref{eq:objective_IFH_output} not only depends on his control strategy $u_i$ but also on other players control strategies $u_{-i}$ indirectly through the state equation \eqref{eq:state_dynamics}. This implies that the players costs are interdependent. As a result, we are in the framework of a game played over the infinite horizon $[0,\infty)$. In what follows, we define the joint control $u(t) := \begin{bmatrix}
u_1^\prime(t) & u_2^\prime(t) & \cdots & u_N^\prime(t)
\end{bmatrix} ^\prime
\in \mathsf U_1 \times \mathsf U_2 \times \cdots \times \mathsf U_N
=: \mathsf U $. Using \eqref{eq:fbcontrol}, the joint feedback strategy is 
\begin{align}
     u(t) =
\begin{bmatrix}
(F_1 C_1)^\prime & (F_2C_2)^\prime &
\cdots &(F_N C_N)^\prime
\end{bmatrix}^\prime
x(t) =: F x(t). 
\label{eq:joint_feedbackcontrol}
\end{align}
Next, we consider a set of parameters denoted as $ \tb{\mathcal{A}} := \left \{ (\alpha_1, \cdots, \alpha_N)~\mid~ \alpha_i \in (0,1),~\sum_{i \in \mathsf N} \alpha_i = 1 \right \}.
$
Using the individual cost functionals \eqref{eq:objective_IFH_output}, the joint weighted cost functional of the differential game is defined as follows 
\begin{align}
    J_{\alpha} (x_0, u) := \sum_{i \in \mathsf{N}}\alpha_i J_i= \int_{0}^{\infty} \left(x(t)^\prime Q_{\alpha}x(t) + u(t)^\prime R_{\alpha} u(t) \right)dt , 
  \label{eq:joint_weighted_cost}
\end{align}
where $(\alpha_1, \cdots, \alpha_N):= \mathbf{\alpha} \in \tb{\mathcal{A}}$, $Q_{\mathbf{\alpha}}=\sum_{i \in \mathsf{N}} C_i^\prime (\alpha_i Q_i) C_i \in \mathbb{R}^{n \times n}$ and $R_{\alpha}= \oplus_{i\in \mathsf N} \alpha_i R_i \in \mathbb{R}^{m \times m}$, and $m:= m_1 + m_2 + \cdots + m_N$. Let $B := [
    B_1 ~B_2 ~\cdots ~B_N ] \in \mathbb{R}^{ m \times n}$. Then, the state dynamics in \eqref{eq:state_dynamics} can be rewritten using the control \eqref{eq:joint_feedbackcontrol}
\begin{align}
     \dot{x}(t)= Ax(t) + Bu(t), ~ x(0)=x_0, 
     \label{eq:state_dynamics_cooperation}
\end{align}
where the state $x(t)$ evolves according to $\dot{x}(t) = (A+BF)x(t),~x(0) = x_0$. 
We assume that pair $(A,B)$ is stabilizable. Next, we recall the Pareto optimality 
definition 
\cite{Engwerda:05}.
\begin{definition}[Pareto optimality]
\label{def:pareto}
The set of controls $u^{\star} \in \mathsf{U}$ 
is called Pareto optimal if the set of 
inequalities $J_i(x_0, u) \leq J_i(x_0, u^{\star})$, 
$i=1,2,\ldots, N$, with at least one of the 
inequalities being strict, does not allow for 
any solution $u \in \mathsf{U}$. The 
corresponding point $(J_1(x_0, u^{\star}), 
J_2(x_0, u^{\star}), \ldots, 
J_N(x_0,u^{\star})) \in \mathbb{R}^{N}$ is 
called a Pareto optimal value, and the set of 
all such Pareto optimal values is called the 
Pareto frontier.
\end{definition}
\tb{The following classical result provides a 
sufficient condition for Pareto optimality; see \cite{Leitmann:74}.}
\begin{lemma}
\label{lem:pareto_optimality_sufficient}
Let $\alpha \in \mathcal{A}$. If $u^{\star} \in 
\arg\min_{u \in \mathsf{U}} \left\{ \sum_{i=1}^{N} 
\alpha_i J_i(u) \right\}$, then $u^{\star}$ is 
Pareto optimal.
\end{lemma}
\tb{Lemma~\ref{lem:pareto_optimality_sufficient} 
states that any minimizer of the weighted sum 
$\sum_i \alpha_i J_i$ is Pareto optimal. The 
connection between Pareto optimality and 
weighted-sum optimization is well established 
in the CDG literature; see, e.g., 
\cite{Engwerda:08CDG, Reddy:2013Automatica}.}
The weighted-sum optimization problem in 
Lemma~\ref{lem:pareto_optimality_sufficient} for the LQ-CDG~\eqref{eq:joint_weighted_cost}--\eqref{eq:state_dynamics_cooperation} 
is formulated as follows 
% \begin{align}
%     \mathrm{PO}:~\arg \min_{u \in \mathsf{U}}~ 
%     J_{\alpha}(x_0, u)~\mathrm{s.t.}~
%     \dot{x}(t) = Ax(t) + Bu(t),~ x(0) = x_0.
%     \label{eq:PO}
% \end{align}
\begin{align}
    \mathrm{PO}: \quad \arg \min_{u \in \mathsf{U}} \quad & J_{\alpha}(x_0, u) \label{eq:PO} \\
    \text{subject to} \quad & \dot{x}(t) = Ax(t) + Bu(t), ~ x(0)=x_0. \notag 
\end{align}
Following \cite{Engwerda:05} (see also \cite{Kazerooni:09}), by 
varying $\alpha \in \mathcal{A}$, the Pareto 
optimal controls $u^{\star}(t) = 
-R_{\alpha}^{-1} B^\prime P_{\alpha} x(t) 
=: F^{\star}x(t)$ \footnote{Note 
that $F^\star$ depends on $\alpha \in \mathcal{A}$ through the 
solution $P_\alpha$ of $\mathrm{ARE}_\alpha$. 
For brevity, we write $F^\star$ instead of 
$F^\star(\alpha)$.} are obtained from the 
weighted-sum optimal control problem \eqref{eq:PO}, 
where $P_{\alpha} \succ 0$ is the stabilizing 
solution of the following algebraic Riccati 
equation 
\begin{align}
    \mathrm{ARE}_{\alpha}:~A^\prime P_{\alpha} + 
    P_{\alpha} A + Q_{\alpha} - P_{\alpha} B 
    R_{\alpha}^{-1} B^\prime P_{\alpha} = 0. 
    \label{eq:ARE_alpha}
\end{align}
Under full state feedback ($C_i = I_n, \forall i \in \mathsf{N}$), 
the Pareto optimal strategy 
$F^\star = \begin{bmatrix}
    {F_1^{\star}}^\prime &{F_2^{\star}}^\prime &\cdots &{F_N^{\star}}^\prime
\end{bmatrix}^\prime = -R_\alpha^{-1} B^\prime P_\alpha$ directly implements the solution of problem in \eqref{eq:PO}. However, under the output feedback 
assumption in \eqref{eq:fbcontrol}, each 
player $i \in \mathsf{N}$ can only implement 
$u_i(t) = F_i y_i(t) = F_i C_i x(t)$. Following 
\cite[Theorem~2.1 and Remarks~2.2--2.3]{Engwerda:08}, 
the output feedback strategies $(F_i, F_{-i})$ 
that achieve Pareto optimality must satisfy 
$F^{\star} = \begin{bmatrix} (F_1^{\star} C_1)^\prime &(F_2^{\star} C_2)^\prime & \cdots & 
(F_N^{\star} C_N)^\prime \end{bmatrix}^\prime = 
-R_{\alpha}^{-1} B^\prime P_{\alpha}$. From here, the 
output feedback Pareto optimal strategy $F_i^{\star}$ for 
player $i \in \mathsf{N}$ is computed by solving
\begin{align}
    F_i^{\star} C_i = -G_i R_{\alpha}^{-1} B^\prime 
    P_{\alpha}, \label{eq:FiCi}
\end{align}
where $G_i := \begin{bmatrix} 0_{m_i \times m_1} 
& \cdots & I_{m_{i}} & \cdots & 
0_{m_i \times m_{N}} \end{bmatrix}$. 
Post-multiplying both sides of \eqref{eq:FiCi} with $C_i^\prime$ and we obtain $F_i^{\star} C_i C_i^\prime = -G_i R_{\alpha}^{-1} B^\prime 
P_{\alpha} C_i^\prime$. As $\text{rank}(C_i)=s_i\leq n$, we have that $C_iC_i^\prime$ is invertible, and this implies $F_i^{\star} = -G_i R_{\alpha}^{-1} B^\prime P_{\alpha} C_i^\prime \left(C_i C_i^\prime\right)^{-1}$, which further implies $F_i^{\star}C_i = -G_i R_{\alpha}^{-1} B^\prime P_{\alpha} 
    C_i^\prime \left(C_i C_i^\prime\right)^{-1}C_i$. So, the Pareto optimal strategies under \eqref{eq:fbcontrol} are solvable if the 
stabilizing solution $P_{\alpha} \succ 0$ of 
$\mathrm{ARE}_{\alpha}$ satisfies
\begin{align}
   \mathrm{SC_1}:~ G_i R_{\alpha}^{-1} B^\prime 
   P_{\alpha} \bigl( I_{n} - C_i^\prime 
   (C_i C_i^\prime)^{-1} C_i \bigr) = 0,~
   \forall i \in \mathsf{N}. 
   \label{eq:struct_condn1}
\end{align}
Thus, the set of output feedback Pareto optimal controls 
exists if, for all $\alpha \in \mathcal{A}$, the 
stabilizing solution $P_{\alpha} \succ 0$ of 
\eqref{eq:ARE_alpha} satisfies 
\eqref{eq:struct_condn1}. Note that, for the full 
state feedback case $C_i = I_n$ for all 
$i \in \mathsf{N}$, \eqref{eq:struct_condn1} is 
trivially satisfied since 
$(I_n - C_i^\prime(C_i C_i^\prime)^{-1}C_i) = 
0_{n \times n}$.
\begin{remark}
\label{rem:SC1_SC2}
By substituting $F = -R_{\alpha}^{-1} B^\prime 
P_{\alpha}$ into \eqref{eq:struct_condn1}, 
$\mathrm{SC}_1$ can be equivalently rewritten as
\begin{align}
   \mathrm{SC_2}:~ G_i F \bigl(I_{n} - C_i^\prime 
   (C_i C_i^\prime)^{-1} C_i \bigr) = 0,~
   \forall i \in \mathsf{N}. 
   \label{eq:struct_condn2}
\end{align}
Note that $\mathrm{SC}_1$ is expressed using $P_\alpha$, 
while $\mathrm{SC}_2$ is expressed directly in 
terms of $F$. Both ensure that $G_i F$ lies in the 
row space of $C_i$, so that player $i$'s control depends only on $y_i$. If there exists $F \in \mathbb{R}^{m \times n}$ 
satisfying \eqref{eq:struct_condn2}, then $F$ is 
said to be structured and is denoted by 
$F_{\mathrm{s}}$, and the control law 
$u(t) = F_{\mathrm{s}} x(t)$ in 
\eqref{eq:joint_feedbackcontrol} is referred to 
as a structured feedback control law, which implements a feedback controller while obeying the output feedback information structure in \eqref{eq:fbcontrol}.
\end{remark}
\begin{remark}
\tb{The matrix $(I_n - C_i^\prime 
(C_i C_i^\prime)^{-1} C_i)$ in 
\eqref{eq:struct_condn1} and \eqref{eq:struct_condn2} is the 
orthogonal projection onto the null space of 
$C_i,~i \in \mathsf{N}$. Both \eqref{eq:struct_condn1} and \eqref{eq:struct_condn2} require that player $i$'s control action has zero component in the state directions not accessible through \eqref{eq:output_vector}, which ensures implementability using output feedback \eqref{eq:fbcontrol}.}
\end{remark}
\tb{In the following example, we illustrate that computing a Pareto optimal solution under output feedback is difficult. Specifically, obtaining a stabilizing solution 
$P_{\alpha} \succ 0$ from \eqref{eq:ARE_alpha} 
that satisfies the condition 
\eqref{eq:struct_condn1} is too restrictive. 
This may not be possible even in low-dimensional 
differential games.}  
\begin{example} 
\label{exp:non_existence}
We consider a two-player game with output feedback information. The game \eqref{eq:LQDG} parameters are $A = \left[\begin{smallmatrix}
    0 &1 \\-1 &-2
\end{smallmatrix} \right]$, $B_1 = \left[\begin{smallmatrix}
    1 &0
\end{smallmatrix} \right]^\prime$, $B_2 = \left[\begin{smallmatrix}
0 &1 \end{smallmatrix} \right]$, $C_1 = \left[\begin{smallmatrix}
    1 &0
\end{smallmatrix} \right]$, 
$C_2 =  \left[\begin{smallmatrix} 0 &1 
\end{smallmatrix} \right]$, $Q_1 = 1$, $Q_2 = 5$, $R_1 = 1$, and $R_2 = 2.5$. 
\tb{By varying $\alpha \in \mathcal{A}$, none of 
the stabilizing solutions $P_\alpha \succ 0$ of 
\eqref{eq:ARE_alpha} satisfies 
\eqref{eq:struct_condn1}. Consequently, it 
remains unclear whether output feedback Pareto 
optimal solutions exist for this game.}
\end{example}
\subsection{Problem statement}
Given the computational difficulty associated with output feedback Pareto optimal controls, we propose the notion of feedback GCSC for LQ-CDG with NTU. Inspired from suboptimal control, instead of minimizing 
the weighted team cost, the players seek a structured feedback strategy that guarantees the weighted team cost does not exceed a given threshold. 
%
% \begin{problem}[$\mathrm{GC}$] 
% \label{prob:GC}
% Consider the LQ-CDG defined by 
% \eqref{eq:joint_weighted_cost}--\eqref{eq:state_dynamics_cooperation} 
% with output feedback information structure 
% \eqref{eq:fbcontrol}. Given 
% $\alpha \in \mathcal{A}$ and $\delta > 0$, 
% find a structured feedback gain 
% $F_{\mathrm{s}} \in \mathbb{R}^{m \times n}$ 
% satisfying $\mathrm{SC}_2$ in 
% \eqref{eq:struct_condn2} such that 
% $J_{\alpha}(F_{\mathrm{s}}) < \delta$ for all 
% $x_0 \in \mathcal{B}(r)$ and 
% $A + BF_{\mathrm{s}}$ is Hurwitz.
% \end{problem}

\begin{problem}[$\mathrm{GC}$] 
\label{prob:GC}
Consider the LQ-CDG defined by \eqref{eq:joint_weighted_cost}--\eqref{eq:state_dynamics_cooperation} with the output feedback information structure \eqref{eq:fbcontrol}. Let $r>0$, $\alpha \in \mathcal{A}$, and $\delta>0$ be given. Find a structured feedback gain 
$F_{\mathrm{s}} \in \mathbb{R}^{m \times n}$ satisfying the structural condition 
$\mathrm{SC}_2$ in \eqref{eq:struct_condn2} such that
\begin{subequations}
\label{eq:GCproblem}
\begin{align}
    & J_{\alpha}(F_{\mathrm{s}}) < \delta,~ \forall x_0 \in \mathcal{B}(r), \label{eq:prob_cond1} \\
    & A + BF_{\mathrm{s}} \text{ is Hurwitz.} \label{eq:prob_cond2}
\end{align}
\end{subequations}
\end{problem}
% \tb{
% \begin{equation}
% \begin{aligned}
% \text{GC}:~~\text{find} \quad
% & F_{\mathrm{s}} \in \mathbb{R}^{m \times n} \\
% \text{subject to}~~
% & J_{\alpha}(F_{\mathrm{s}}) < \delta, ~~ \forall x_0 \in \mathcal{B}(r), \\
% & F_{\mathrm{s}} \text{ satisfies } \mathrm{SC}_2
% \text{ in \eqref{eq:struct_condn2}}, \\
% & A+BF_{\mathrm{s}} \text{ is Hurwitz.}
% \end{aligned}
% \end{equation}
% }
%
\section{Guaranteed Cost Structured Control} 
\label{sec:OFGCC}
In the section, we first define the guaranteed cost controller set as follows.
\begin{definition}
\label{def:guaranteed_cost_set}
Let $r > 0$ be given and let $x_0 \in \mathcal{B}(r)$ be an initial condition. Let $\delta > 0$ and $\mathbf{\alpha} \in \tb{\mathcal{A}}$. The guaranteed cost controller set is defined as 
\begin{align}
    \mathfrak g_{\mathbf{\alpha}}(\delta)
:= \left\{F \in \mathbb{R}^{m\times n} ~\big|~
J_{\mathbf{\alpha}}(F) < \delta,~\forall x_0 \in \mathcal{B}(r)
\right\}.
\label{eq:g_delta}
\end{align}
\end{definition}
Using the set in \eqref{eq:g_delta}, we 
introduce the GCSC set. 
\tb{\begin{definition}
\label{def:guaranteed_cost_structured_control}
Let $r > 0$ be given and let $x_0 \in \mathcal{B}(r)$ 
be an initial condition. Let $\delta > 0$ and 
$\alpha \in \mathcal{A}$. The GCSC set is defined as 
\begin{align}
  \hat{\mathfrak{g}}_{\alpha}(\delta) :=
  \left\{ F_{\mathrm{s}} \in \mathbb{R}^{m \times n}
  ~\big|~
  F_{\mathrm{s}} \in \mathfrak{g}_{\alpha}(\delta)
  \text{ and }
  F_{\mathrm{s}} \text{ satisfies } 
  \eqref{eq:struct_condn2}
  \right\}. 
\label{eq:set_GC}
\end{align}
A feedback strategy $F_{\mathrm{s}}$ is called a 
feedback GCSC if 
$F_{\mathrm{s}} \in \hat{\mathfrak{g}}_{\alpha}(\delta)$. 
$F_{\mathrm{s}}$ is said to be stabilizing if 
$A + BF_{\mathrm{s}}$ is Hurwitz.
\end{definition}
}
Next we define the admissible weight set. 
\begin{definition}
\label{def:admissible_alpha_set}
For a given $\delta>0$, the admissible weight set is
defined as folllows 
\begin{align}
\tb{\mathcal{A}}(\delta) := \left\{\mathbf{\alpha}\in \tb{\mathcal{A}} ~\big|~\hat{\mathfrak{g}}_{\mathbf{\alpha}}(\delta)\neq\emptyset
\right\}.
\label{eq:weighted_admissible_set}
\end{align}
\end{definition}
In the following result, we establish the monotonicity properties of the sets defined in Definitions \ref{def:guaranteed_cost_structured_control} and \ref{def:admissible_alpha_set}.
\begin{lemma}
\label{lem:monotonicity_properties}
Let $0 < \delta \le \bar{\delta}$. Then the following properties hold 
\begin{enumerate}
    \item For any $\mathbf{\alpha} \in \tb{\mathcal{A}}$, $\hat{\mathfrak g}_{\mathbf{\alpha}}(\delta)
    \subseteq
    \hat{\mathfrak g}_{\mathbf{\alpha}}(\bar{\delta}).$
    \item The admissible weight sets satisfy $ \tb{\mathcal{A}}(\delta) \subseteq \tb{\mathcal{A}}(\bar{\delta}).$
\end{enumerate}
\end{lemma}
\vspace{-2mm}
\begin{proof}
(1) The result follows from Definition~ \ref{def:guaranteed_cost_structured_control}. If $F_{\mathrm{s}} \in \hat{\mathfrak g}_{\mathbf{\alpha}}(\delta)$, then 
$J_{\mathbf{\alpha}}(F_{\mathrm{s}}) < \delta$. 
Since $\delta \le \bar{\delta}$, it follows that 
$J_{\mathbf{\alpha}}(F_{\mathrm{s}}) < \bar{\delta}$, and hence $F_{\mathrm{s}} \in \hat{\mathfrak g}_{\mathbf{\alpha}}(\bar{\delta})$.

(2) Let $\mathbf{\alpha} \in \tb{\mathcal{A}}(\delta)$. 
From Definition \ref{def:admissible_alpha_set}, there exists $F_\mathrm{s}$ such that
$J_{\mathbf{\alpha}}(F_\mathrm{s}) < \delta$. 
Since $\delta \le \bar{\delta}$, we have 
$J_{\mathbf{\alpha}}(F_\mathrm{s}) < \bar{\delta}$, and therefore 
$\mathbf{\alpha} \in \tb{\mathcal{A}}(\bar{\delta})$.
\end{proof}
\begin{remark}
The monotonicity of the set $\tb{\mathcal{A}}(\delta)$ in Lemma \ref{lem:monotonicity_properties} has a clear geometric interpretation. The subset $\tb{\mathcal{A}}(\delta) \subset \tb{\mathcal{A}}$ consists of weight vectors $\alpha$ for which there exists a structured controller $F_{\mathrm{s}}$ satisfying $J_{\mathbf{\alpha}}(F_{\mathrm{s}}) < \delta$. For small $\delta$, the upper bound is restrictive and $\tb{\mathcal{A}}(\delta)$ may be empty. As $\delta$ increases, the bound becomes less restrictive and $\tb{\mathcal{A}}(\delta)$ expands monotonically within $\tb{\mathcal{A}}$. 
\end{remark}
\begin{remark}
  For a given $\delta > 0$, the set $\tb{\mathcal{A}}(\delta) \subset \tb{\mathcal{A}}$ may contain multiple weight vectors. 
  % That is, there may exist several $\alpha \in \tb{\mathcal{A}}(\delta)$ for which $\hat{\mathfrak{g}}_{\mathbf{\alpha}}(\delta)$ is non-empty. 
  The selection of a particular $\alpha \in \tb{\mathcal{A}}(\delta)$ can be determined using bargaining solutions (see \cite[Ch. 6]{Engwerda:05}). In particular, for the two-player case, the Nash bargaining solution can be used to select a weight vector within our framework. The bargaining problem for more than two players is beyond the scope of this work. The present work focuses on synthesizing a structured controller $F_\mathrm{s}$ for a given admissible weight vector $\alpha \in \tb{\mathcal{A}}(\delta)$ that achieves $J_{\alpha}(F_{\mathrm{s}}) < \delta$. 
\label{rem:selection_weight}
\end{remark}
\tb{The following result shows that the set of Pareto optimal controls is contained in the set of feedback GCSCs.}
\begin{theorem}
\label{thm:pareto_gc}
\tb{Consider problem $\mathrm{PO}$ in \eqref{eq:PO}. Let $\alpha \in \mathcal{A}$ and suppose that the 
stabilizing solution $P_\alpha \succ 0$ of 
$\mathrm{ARE}_\alpha$ in \eqref{eq:ARE_alpha} 
satisfies $\mathrm{SC}_1$ in 
\eqref{eq:struct_condn1}. Let 
$F^\star = -R_\alpha^{-1} B^\prime P_\alpha$ 
denote the corresponding output feedback 
Pareto optimal gain. Then, for any 
$\rho > 0$, setting 
$\delta := J_\alpha(F^\star) + \rho$, 
we have $F^\star \in \hat{g}_\alpha(\delta)$. 
Equivalently, $F^\star \in \hat{g}_\alpha(\delta)$ 
for every $\delta > J_\alpha(F^\star)$.} 
\end{theorem}
\vspace{-2mm}
\begin{proof}
Follows from Lemma~\ref{lem:monotonicity_properties}.
\end{proof} 
In what follows, we quantify the performance of the proposed feedback GCSC strategy relative to both Pareto optimal controls (if they exist) and output feedback optimal control. Now, we define the performance measure ratios 
\begin{align}
    \eta_1 \tb{(F^{\star})} := \frac{J_{\mathrm{PO}}\tb{(F^{\star})}}{J_{\mathrm{OPT}}}, 
    \qquad 
    \eta_2 \tb{(F_{\mathrm{s}})} := \frac{J_{\mathrm{GC}} \tb{(F_{\mathrm{s}})}}{J_{\mathrm{OPT}}},
\label{eq:eta1_eta2}
\end{align}
\tb{where $J_{\mathrm{PO}}(F^{\star}) := \sum_i J_i(F^{\star})$ denotes the total cost corresponding to a Pareto optimal solution (if it exists) of problem $\mathrm{PO}$, $J_{\mathrm{GC}}(F_{\mathrm{s}}) := \sum_i J_i(F_{\mathrm{s}})$ denotes the total cost achieved by the feedback GCSC strategy obtained from problem $\mathrm{GC}$,} and $J_{\mathrm{OPT}}$ denotes the optimal cooperative cost. The cooperative optimal cost $J_{\mathrm{OPT}}$ is obtained from the infinite-horizon output feedback linear quadratic problem \cite{Engwerda:08}: 
% \begin{align}
%     \text{OPT}:~J_{OPT}= \min_{(F_i, F_{-i})}
% 	\sum_{i=1}^N J_i(F_i,F_{-i}), ~\text{s.t.}~\dot{x}(t)=(A+\sum_{i\in \mathsf N} B_i F_i C_i)x(t),~x(0)=x_0.
% \end{align}
\begin{align}
    \text{OPT}: ~J_{\text{OPT}} = \min_{( F_i, F_{-i})} ~ & \sum_{i=1}^N J_i(x_0, F_i,F_{-i}) \label{eq:OPT} \\
    \text{subject to} ~& \dot{x}(t) = \Big(A + \sum_{i \in \mathsf{N}} B_i F_i C_i \Big) x(t),~x(0)=x_0 \notag. 
\end{align}
Let $\tilde{Q}:=\sum_{i\in \mathsf N} C_i^\prime Q_i C_i$, $\tilde{R}:=\oplus_{i\in \mathsf N}R_i$, $\tilde{B}:=[B_1~B_2~\cdots~B_N]$. Since $\tilde{Q}\succeq 0$ and $\tilde{R}\succ 0$, \tb{the optimal cooperative cost of the problem 
$\mathrm{OPT}$ is given by 
$J_{\mathrm{OPT}} = x_0^\prime P x_0$, where 
$P \succ 0$ is the unique stabilizing solution of $A^\prime P + PA + \tilde{Q} - 
P\tilde{B}\tilde{R}^{-1}\tilde{B}^\prime P = 0$; 
see~\cite[Ch.~5]{Engwerda:05}.} \tb{It is well known that $J_{\mathrm{OPT}}= x_0^\prime P x_0$ provides a lower bound on any achievable cooperative cost \cite{Engwerda:08}. }
%However, in distributed control of MAS, output feedback strategies that attain this bound may not exist; see \cite[Remark 4.9]{Roy:2025DGAA}. 
\tb{In the following result, we establish the lower bounds on performance measures $\eta_1 (F^{\star})$ and $\eta_2 (F_{\mathrm{s}})$.}
\tb{\begin{lemma}
\label{lem:eta1_eta2_bound}
\begin{enumerate}
    \item Let $\alpha \in \mathcal{A}$ and consider problem $\mathrm{PO}$.  Suppose that the stabilizing solution $P_\alpha \succ 0$ of $\mathrm{ARE}_\alpha$ in \eqref{eq:ARE_alpha} satisfies $\mathrm{SC}_1$ in \eqref{eq:struct_condn1}. Let $ F^\star = -R_\alpha^{-1} B^\prime P_\alpha$ denote the corresponding output feedback Pareto optimal gain. Then $\eta_1(F^{\star})\geq 1$.
      \item Let $\alpha \in \mathcal{A}$ and consider Problem \ref{prob:GC}. Let $\delta > 0$ be given. Suppose that a structural feedback gain $F_s$ exists and satisfies $\mathrm{SC}_2$ in \eqref{eq:struct_condn2}. Then $\eta_2 (F_{\mathrm{s}}) \geq 1$.
\end{enumerate}
\end{lemma} }
\vspace{-2mm}
\begin{proof}
% (1) Since $J_{\mathrm{OPT}}$ denotes the minimum achievable cost obtained through optimal control problem and a Pareto optimal solution exists by assumption, we have $J_{\mathrm{OPT}} \leq J_{\mathrm{PO}}$, which implies $\eta_1 \geq 1$.

% (2) Similarly, since $J_{\mathrm{OPT}}$ denotes the minimum achievable cost and a GCSC solution exists by assumption, we obtain $J_{\mathrm{OPT}} \leq J_{\mathrm{GC}}$, which implies $\eta_2 \geq 1$.
% Since $J_{\mathrm{OPT}} = x_0^\prime P x_0$ denotes the minimum achievable cooperative optimal cost of problem $\mathrm{OPT}$, it follows from \cite[Theorem~2.1 and Remarks~2.2--2.3]{Engwerda:08} (see also \cite[Remark~4.9]{Roy:2025DGAA}) that output feedback strategies attaining this bound may not exist. As a result, we have $x_0^\prime P x_0 \leq J_{\mathrm{PO}} (F^{\star})$ and $x_0^\prime P x_0 \leq J_{\mathrm{GC}} (F_{\mathrm{s}})$, whenever the corresponding Pareto optimal and GCSC solutions exist. Hence, $\eta_1 (F^{\star}) \geq 1$ and $\eta_2 (F_{\mathrm{s}})\geq 1.$
\tb{Since $J_{\mathrm{OPT}} = x_0^\prime P x_0$ 
denotes the minimum achievable cooperative 
cost of problem \eqref{eq:OPT}, it follows 
from \cite[Theorem~2.1 and 
Remarks~2.2--2.3]{Engwerda:08} (see also 
\cite[Remark~4.9]{Roy:2025DGAA}) that output 
feedback strategies attaining this bound may 
not exist. As a result, 
$x_0^\prime P x_0 \leq 
J_{\mathrm{PO}}(F^{\star})$ and 
$x_0^\prime P x_0 \leq 
J_{\mathrm{GC}}(F_{\mathrm{s}})$, whenever 
the corresponding Pareto optimal and GCSC 
solutions exist. Hence, $\eta_1(F^{\star}) 
\geq 1$ and $\eta_2(F_{\mathrm{s}}) \geq 1$.}
\end{proof}
\vspace{-2mm}
The following result provides an upper bound on $\eta_2$.
\begin{corollary}
\label{cor:eta2_upperbound}
\tb{Consider Problem~\ref{prob:GC}.} Let $\delta > 0$ and consider uniform weights 
$\alpha_i = \frac{1}{N}$ for all $i \in \mathsf{N}$. 
If the feedback GCSC satisfies 
$\sum_{i=1}^{N} \alpha_i J_i \tb{(F_{\mathrm{s}})} < 
\delta$, then the total team cost satisfies 
$\sum_{i=1}^{N} J_i \tb{(F_{\mathrm{s}})} < N\delta$. 
Consequently, $1 \leq \eta_2 \tb{(F_{\mathrm{s}})} < 
\frac{N\delta}{J_{\mathrm{OPT}}}$.
\end{corollary}
\begin{remark}
\tb{The performance measures $\eta_1$ and $\eta_2$ are inspired by efficiency-gap measures such as the Price of Stability (PoS) in \cite[Theorem~4.6]{Roy:2025DGAA}; see also \cite{Anshelevich:2008pos}. However, unlike the PoS, which quantifies inefficiency arising from equilibrium outcomes in non-cooperative games, the measures $\eta_1$ and $\eta_2$ have fundamentally different interpretations in the present cooperative framework. Specifically, $\eta_1$ and $\eta_2$ quantify the performance degradation caused by structural constraints, whenever Pareto optimal controls and feedback GCSC exist, respectively, relative to the optimal cooperative cost $J_{\mathrm{OPT}}$.
} 
\end{remark}
\section{Verification and synthesis of feedback guaranteed cost structured control}
\label{sec:verfication_existence}
In this section, we provide verification and synthesis results for a feedback GCSC. For a given $\delta > 0$, the following result establishes sufficient conditions under which a feedback strategy is a GCSC.
\begin{theorem}
\label{thm:GC_verification}
% Consider the system \eqref{eq:state_dynamics_cooperation} with cost functional \eqref{eq:joint_weighted_cost}. 
Consider Problem \ref{prob:GC}. Let $r > 0$ be given and let $x_0 \in \mathcal{B}(r)$ be an initial condition.  Let $\delta > 0$ and $\mathbf{\alpha} \in \tb{\mathcal{A}}$ be given.
Suppose that a structured strategy $F_{\mathrm{s}}$ is given that satisfies \eqref{eq:struct_condn2}. 
If there exists $P_{\mathbf{\alpha}} \succ 0$ such that
\begin{subequations}
\label{eq:gc_test_LMI}
\begin{align}
& A_{\mathrm{cl}}(F_{\mathrm{s}})^\prime P_{\mathbf{\alpha}}
+ P_{\mathbf{\alpha}} A_{\mathrm{cl}}(F_\mathrm{s})
+ Q_{\mathbf{\alpha}}
+ {F_\mathrm{s}}^\prime R_{\mathbf{\alpha}}
F_{\mathrm{s}} \prec 0, \label{eq:gc_test_LMI1} \\
& {x_0}^\prime P_{\mathbf{\alpha}} x_0 < \delta, \label{eq:gc_test_LMI2}
\end{align}
\end{subequations}
where $A_{\mathrm{cl}}(F_{\mathrm{s}})
:= A + B F_{\mathrm{s}}$. Then, $F_{\mathrm{s}}$ is a stabilizing GCSC.
\end{theorem}
\vspace{-2mm}
\begin{proof}
Since $R_{\mathbf{\alpha}} \succ 0$ and 
$\bar{Q}_{\mathbf{\alpha}}:=Q_{\mathbf{\alpha}}
+ {F_\mathrm{s}}^\prime R_{\mathbf{\alpha}}
F_{\mathrm{s}} \succeq 0$, the existence of 
$P_{\mathbf{\alpha}} \succ 0$ satisfying 
\eqref{eq:gc_test_LMI1} implies that 
$A_{\mathrm{cl}}(F_{\mathrm{s}})$ is stable; 
see \cite[Theorem~4]{Jiao:2019TAC}. Solving 
$\dot{x}(t)=A_{\mathrm{cl}}x(t)$ with 
$x(0)=x_0 \in \mathcal{B}(r)$ gives 
$x(t)=e^{A_{\mathrm{cl}}t}x_0$, and hence
$J_{\mathbf{\alpha}}(F_{\mathrm{s}})
=x_0^\prime Y_{\mathbf{\alpha}}x_0$, where $Y_{\mathbf{\alpha}}
=\int_0^\infty e^{A_{\mathrm{cl}}^\prime t}
\bar{Q}_{\mathbf{\alpha}}
e^{A_{\mathrm{cl}}t}~dt$ is the unique positive semi-definite solution of $A_{\mathrm{cl}}^\prime Y_{\mathbf{\alpha}}
+Y_{\mathbf{\alpha}}A_{\mathrm{cl}}
+\bar{Q}_{\mathbf{\alpha}}=0$; see \cite[Lemma~3]{Jiao:2019TAC}. Defining 
$X_{\mathbf{\alpha}}:=P_{\mathbf{\alpha}}
-Y_{\mathbf{\alpha}}$ and using 
\eqref{eq:gc_test_LMI1}, we obtain
$A_{\mathrm{cl}}^\prime X_{\mathbf{\alpha}}
+X_{\mathbf{\alpha}}A_{\mathrm{cl}}\prec0$, 
which together with the stability of 
$A_{\mathrm{cl}}$ implies 
$X_{\mathbf{\alpha}}\succ0$. Hence, 
$Y_{\mathbf{\alpha}}\preceq P_{\mathbf{\alpha}}$, 
and from \eqref{eq:gc_test_LMI2}, we have $J_{\mathbf{\alpha}}(F_{\mathrm{s}})
=x_0^\prime Y_{\mathbf{\alpha}}x_0
\leq x_0^\prime P_{\mathbf{\alpha}}x_0
<\delta.$
\end{proof}
\begin{remark}
\tb{Theorem~\ref{thm:GC_verification} provides a verification 
condition for a feedback GCSC strategy in an LQ-CDG, whereas \cite[Theorem~5.4]{Roy:2025DGAA} 
provides verification of an output feedback guaranteed 
cost equilibrium in an LQ-NCDG. We note that Theorem~\ref{thm:GC_verification} requires verifying $2$ linear matrix inequalities (LMIs), while \cite[Theorem~5.4]{Roy:2025DGAA} requires verifying $2N$ LMIs. }
\end{remark}
Note that \eqref{eq:gc_test_LMI1}--\eqref{eq:gc_test_LMI2} 
are LMIs in $P_{\alpha} \succ 0$. Hence, using LMI solvers, we can verify whether a given structured strategy is a feedback GCSC. However, this does not provide a synthesis method. \tb{The synthesis of a feedback GCSC depends on the non-emptiness of the set \eqref{eq:set_GC}. 
The following theorem provides sufficient conditions under which the set \eqref{eq:set_GC} is non-empty.}
\begin{lemma}
\label{lem:existence_cop_gc}
%Consider the system \eqref{eq:state_dynamics_cooperation} with cost functional \eqref{eq:joint_weighted_cost}. 
\tb{Consider Problem~\ref{prob:GC}.} 
Let $r > 0$ be given and let $x_0 \in \mathcal{B}(r)$ be an initial condition. 
Let $\delta > 0$ and $\mathbf{\alpha} \in \tb{\mathcal{A}}$ be given. 
Suppose there exists a pair $(P_{\alpha}, F)$ with
$P_{\alpha} \succ 0$ and $F \in \mathbb{R}^{m \times n}$ such that
\begin{subequations}
\label{eq:LMI_gc_existence}
\begin{align}
&(A + B F)^\prime P_{\alpha}
+ P_{\alpha} (A + B F)
+ Q_{\mathbf{\alpha}}
+ {F}^\prime R_{\mathbf{\alpha}}
F \prec 0, \label{eq:LMI_existence1} \\
& x_0^\prime P_{\mathbf{\alpha}} x_0 < \delta, \label{eq:LMI_existence2} \\
& G_i F (I_{n}-C_i^{\prime} (C_iC_i^\prime)^{-1}C_i) =0,~\forall i \in \mathsf{N},  
 \label{eq:structural_existence3}
\end{align}
\end{subequations}
where $G_i := \begin{bmatrix}
0_{m_i \times m_1} & \cdots &I_{m_{i}} &\cdots &0_{m_i \times m_{N}} \end{bmatrix}$. Then $\hat{\mathfrak{g}}_{\mathbf{\mathbf{\alpha}}} (\delta) \neq \emptyset$ and $(A+BF)$ is Hurwitz. 
\end{lemma}
\vspace{-2mm}
\begin{proof}
Follows the same lines as that of Theorem \ref{thm:GC_verification}.  
\end{proof}
\vspace{-2mm}
\tb{Note that \eqref{eq:LMI_existence1} is a 
bilinear matrix inequality (BMI) in the variables 
$(P_{\alpha}, F)$, which is non-convex 
\cite{Safonov:96}. Using the projection lemma 
\cite[Lemma 2]{Iwasaki:94} and Schur complement lemma 
\cite[Lemma 2.8]{Duan:13}, this BMI is converted to an 
LMI in the following result.} 
% Note that \eqref{eq:LMI_existence1} is a bilinear matrix inequality (BMI) in the variables $(P_{\alpha}, F)$. It is well known that BMI constraint is non-convex \cite{Safonov:96}. For solvability using LMI solvers, the BMI \eqref{eq:LMI_existence1} must be converted into an LMI constraint. To address this difficulty, in the follwoing result, we use projection lemma \cite{Iwasaki:94} and schur complement lemma \cite{Duan:13}. 
%
% \begin{lemma}(\cite{Iwasaki:94}) 
% \label{lem:Finsler}
% 	Let $  E \in \mathbb R^{n\times m}$, $\text{rank}(E)=m<n$,
% 	$H\in \mathbb R^{s\times n}$, $\text{rank}(H)=s<n$, and 
% 	$\Omega\in \mathbb R^{n\times n}$,~$\Omega=\Omega^\prime$, be given. Then, there exists $F\in \mathbb R^{m\times s}$
% 	satisfying $\Omega+E F H+ (E F H)^\prime  \prec 0$
% 	if and only if $\tb{\mathcal N}_{E^\prime}^\prime \Omega \tb{\mathcal N}_{E^\prime} \prec 0$ and $\tb{\mathcal N}_H^\prime \Omega \tb{\mathcal N}_H \prec 0$ hold, where $\tb{\mathcal N}_{E^\prime}
% 	\in \mathbb R^{n\times (n-m)}$, $\tb{\mathcal N}_H\in \mathbb R^{n\times (n-s)}$, denoting any matrices whose columns form orthonormal
% 	bases of the null spaces of $E^\prime$, $H$ respectively.
% \end{lemma} 
%
\begin{theorem} 
\label{thm:GC_cop_synthesis}
% Consider the system \eqref{eq:state_dynamics_cooperation} with cost functional \eqref{eq:joint_weighted_cost}. 
\tb{Consider Problem~\ref{prob:GC}.} Let $r > 0$ be given and let $x_0 \in \mathcal{B}(r)$ be an initial condition. 
Let $\delta > 0$ and $\mathbf{\alpha} \in \tb{\mathcal{A}}$ be given. Consider the following set
\begin{align}
	\tb{\mathcal{Y}_{\alpha}} &:=\left\{Y \in \mathbb R^{n\times n}~\big|~Y\succ 0, ~\left[\begin{smallmatrix}\delta &x_0^\prime\\x_0&Y\end{smallmatrix} \right]\succ  0,~\tb{\mathcal N_{\hat{B}}}^\prime\Omega_{\alpha}(Y)\tb{\mathcal N_{\hat{B}}}\prec 0 \right\},\label{eq:YLMI}
\end{align} 
    \vspace{-1ex}
    \text{where,} 
    \vspace{1ex}
            \begin{align}
			\Omega_{\alpha} (Y) &:=\begin{bmatrix} 
				YA^\prime + A Y &(\sqrt{Q_\alpha}Y)^\prime& 0_{n\times m }\\
				\sqrt{Q_\alpha}Y& -I_{n}& 0_{ n \times m}\\
				0_{m \times n}&0_{m \times n  }&-R_{\alpha}^{-1} \end{bmatrix},	
		\end{align}	
		and $\tb{\mathcal N_{\hat{B}}}:=\ker \left(\begin{bmatrix}B^\prime & 0_{m \times n}&I_{m}\end{bmatrix}\right)$ denote a matrix with orthonormal columns which span the null spaces of the matrix
		$\hat{B}=\begin{bmatrix} B^\prime & 0_{m \times n} &I_{m} \end{bmatrix}$. Define the set
		\begin{align}
			\tb{\mathcal{P}_{\alpha}} &:=\left\{P_{\alpha} \in \mathbb R^{n\times n} ~|~P_{\alpha} \succ 0, ~P_{\alpha}^{-1}\in \tb{\mathcal{Y}_{\alpha}} \right\}.\label{eq:setPi}
		\end{align}
        Suppose that $\tb{\mathcal{P}_{\alpha}} \neq \emptyset$. For any $P_{\mathbf{\alpha}} \in \tb{\mathcal{P}_{\alpha}}$, if there exists $F \in \mathbb{R}^{m \times n}$ satisfying the LMIs 
		\begin{align} 
			&\begin{bmatrix} 
				(A + BF)^\prime P_{\alpha} +P_{\alpha}(A + BF) &(\sqrt{Q_{\alpha}})^\prime& (\sqrt{R_{\alpha}}F)^\prime\\
				\sqrt{Q_\alpha}& -I_{n} & 0_{n \times m}\\
				\sqrt{R_\alpha}F&0_{m\times n }&-I_{ {m}} \end{bmatrix}\prec 0.
			\label{eq:Feedback_cop_GC} \\
        & G_i F(I_n -C_i^{\prime} (C_iC_i^\prime)^{-1}C_i) =0,~ \forall i \in \mathsf{N}, 
        \label{eq:Feedback_cop_structured}
		\end{align}   
   where $G_i := \begin{bmatrix}
0_{m_i \times m_1} & \cdots &I_{m_{i}} &\cdots &0_{m_i \times m_{N}} \end{bmatrix}$. Then, $F \in \hat{\mathfrak{g}}_{\alpha}(\delta)$, and $(A+BF)$ is Hurwitz. 
\end{theorem}
\vspace{-2mm}
\begin{proof} 
Since $Q_{\mathbf{\alpha}} \succeq 0$, using decomposition  $Q_{\mathbf{\alpha}} = \sqrt{Q_{\mathbf{\alpha}}}.\sqrt{Q_{\mathbf{\alpha}}}$ and applying the Schur complement lemma (see \cite[Lemma 2.8]{Duan:13}) in the BMI \eqref{eq:LMI_existence1} is equivalently written as
\begin{align}
\bar{\Omega}_{\alpha}(P_{\alpha}) + \bar{B}^\prime F \bar{C} + \bar{C}^\prime {F}^\prime \bar{B} \prec 0,
\label{eq:BMI_1}
\end{align}
where,  $\bar{\Omega}_{\alpha} (P_{\mathbf{\alpha}}):=
        \left[\begin{smallmatrix}
         A^\prime P_{\mathbf{\alpha}} + P_{\mathbf{\alpha}} A    &\sqrt{Q}_{\alpha} &0_{n \times m} \\
          \sqrt{Q}_{\alpha} &-I_{n} &0_{n \times m}\\
         0_{m \times n} &0_{n \times m} &-R_{\mathbf{\alpha}}^{-1}
    \end{smallmatrix} \right]$, $\bar{B} =  \left[\begin{smallmatrix}
    B^\prime P_{\mathbf{\alpha}} & 0_{m \times n} &I_{m}
\end{smallmatrix} \right]$, $    
\bar{C} = \left[\begin{smallmatrix}
    I_n & 0_{n \times n} &0_{n \times m}
\end{smallmatrix} \right]$. Using projection lemma (see \cite[Lemma 2]{Iwasaki:94}), the BMI \eqref{eq:BMI_1} in the variables $\left(F, P_{\mathbf{\alpha}}\right)$ is feasible if and only if the following LMIs \eqref{eq:BMI_2} in the variable $P_{\mathbf{\alpha}}$ are feasible.
\begin{align}
{\mathcal{N}_{\bar{C}}}^\prime \bar{\Omega}_{\alpha}(P_{\mathbf{\alpha}}) \tb{\mathcal N_{\bar{C}}} \prec 0, ~{\mathcal{N}_{\bar{B}}}^\prime \bar{\Omega}_{\alpha} (P_{\mathbf{\alpha}}) \tb{\mathcal N_{\bar{B}}} \prec 0.
\label{eq:BMI_2}
\end{align}
Here, $\tb{\mathcal N_{\bar{B}}} :=\ker (\left[\begin{smallmatrix} B^\prime P_{\mathbf{\alpha}} & 0_{m \times n} &I_{m}\end{smallmatrix} \right])$  and is related to $\tb{\mathcal N_{\hat{B}}}$ as $ \tb{\mathcal N_{\bar{B}}}  = \left(P_{\mathbf{\alpha}}^{-1} \oplus I_{n} \oplus I_{m} \right)  \tb{\mathcal N_{\hat{B}}}$. 
Using this relation in \eqref{eq:BMI_2} (as a result, \eqref{eq:LMI_existence1}) is equivalently written as 
\begin{align}
	{\mathcal{N}_{\hat{B}}}^\prime \Omega_{\alpha}(P_{\mathbf{\alpha}}^{-1}) \tb{\mathcal N_{\hat{B}}} \prec 0.
	\label{eq:BMI_3}
\end{align}
Using the Schur complement, the inequality \eqref{eq:LMI_existence2} can be equivalently written as
	\begin{align}
		\left[\begin{smallmatrix}
			\delta &x_0^\prime \\
			x_0 &P_{\mathbf{\alpha}}^{-1}
		\end{smallmatrix} \right]
		\succ 0, 
		\label{eq:LMI_2_ub}
	\end{align}	 
where $x_0 \in \mathcal{B}(r)$. We note that the matrix inequalities \eqref{eq:BMI_3} and \eqref{eq:LMI_2_ub} define the set 
$\tb{\mathcal{P}_{\mathbf{\alpha}}}$ as given in \eqref{eq:setPi}. 
If $\tb{\mathcal{P}_{\mathbf{\alpha}}} \neq \emptyset$, then there exists 
$P_{\mathbf{\alpha}} \succ 0$ such that 
$P_{\mathbf{\alpha}}^{-1} \in \tb{\mathcal{Y}_{\alpha}}$, and hence 
$P_{\mathbf{\alpha}}$ is feasible for \eqref{eq:BMI_3} and \eqref{eq:LMI_2_ub}. It then follows from Lemma~\ref{lem:existence_cop_gc} that, for any 
$P_{\mathbf{\alpha}} \in \tb{\mathcal{P}_{\mathbf{\alpha}}}$, 
if there exists $F \in \mathbb{R}^{m \times n}$ satisfying 
\eqref{eq:Feedback_cop_GC}-\eqref{eq:Feedback_cop_structured}, 
then $F \in \hat{\mathfrak{g}}_{\mathbf{\alpha}}(\delta)$ and the closed-loop system
$A+BF$ is Hurwitz.
\end{proof}
%
%\vspace{-1.9mm}
\begin{remark}
\label{rem:comparision_Roy2025}
\tb{Theorem~\ref{thm:GC_cop_synthesis} and 
\cite[Theorem~5.7]{Roy:2025DGAA} are both 
LMI-based synthesis results that rely on 
fundamental results from the LMI literature, 
namely the projection lemma \cite{Iwasaki:94} and the Schur 
complement lemma \cite{Duan:13}. However, they 
differ in scope and computational method. 
Specifically, Theorem~\ref{thm:GC_cop_synthesis} addresses the synthesis problem in an LQ-CDG framework, whereas \cite[Theorem~5.7]{Roy:2025DGAA} considers an LQ-NCDG setting.
We note that the feasibility conditions in \cite[Theorem~5.7, Eq.~(27e)]{Roy:2025DGAA} must be verified separately for each of the $N$ players. Furthermore, since the resulting feasibility set in \cite[Theorem~5.7, Eq.~(27e)]{Roy:2025DGAA} is non-convex, a semi-definite programming (SDP) relaxation is required; see \cite{Leibfritz:01}. Consequently, the computational procedure depends on the sequential linear programming matrix method (SLPMM); see \cite[Sec.~6.2]{Roy:2025DGAA} for details. In contrast, the set \eqref{eq:setPi} in 
Theorem~\ref{thm:GC_cop_synthesis} is convex. As a result, SDP relaxation and SLPMM are not required. Moreover, in the proposed GCSC framework, the feasibility of set $\eqref{eq:setPi}$ needs to be verified only once.}
%This demonstrates a significant computational 
%advantage of the cooperative framework.}
\end{remark}
\begin{remark}
For computational convenience, we restrict our attention to closed sets. Thus, we consider the following $\epsilon$-approximation of the convex open set \eqref{eq:YLMI}: 
% \begin{align}
%     \tb{\mathcal{Y}^{\epsilon}} := \left \{ Y \in \mathbb{R}^{n \times n} ~\mid~ Y \succ 0,~\left[\begin{smallmatrix}\delta & x_0^\prime \\ x_0 & Y\end{smallmatrix}\right] \succeq \epsilon I_{n+1},~ \tb{\mathcal{N}}_{B}^\prime \Omega(Y) \tb{\mathcal{N}}_{B} \preceq -\epsilon I_{2n+m} \right \}
% \end{align}
\begin{align}
    \tb{\mathcal{Y}_{\alpha}^{\epsilon}} := \Big \{ Y \in \mathbb{R}^{n \times n} ~\big|~ & Y \succ 0,~ \left[\begin{smallmatrix}\delta & x_0^\prime \\ x_0 & Y\end{smallmatrix}\right] \succeq \epsilon I_{n+1}, \notag \\
    & {\mathcal{N}_{\hat{B}}}^\prime \Omega(Y) \mathcal{N}_{\hat{B}} \preceq -\epsilon I_{2n+m} \Big \}. 
\label{eq:set_y_eps}
\end{align}
Using the set in \eqref{eq:set_y_eps}, we obtain the approximate set of \eqref{eq:setPi}, denoted by $\tb{\mathcal{P}_{\mathbf{\alpha}}^{\epsilon}}$. If $\tb{\mathcal{P}_{\mathsf{\alpha}}^{\epsilon}} \neq \emptyset$, then for any feasible $P_{\mathbf{\alpha}} \in \tb{\mathcal{P}_{\mathbf{\alpha}}^{\epsilon}} $, we compute $F_{\mathrm{s}}$ using \eqref{eq:Feedback_cop_GC}-\eqref{eq:Feedback_cop_structured} in Theorem \ref{thm:GC_cop_synthesis}. We then verify whether $F_{\mathrm{s}}$ is a feedback GCSC strategy using Theorem~\ref{thm:GC_verification}.
\label{rem:computation_GCSC}
\end{remark}
\begin{remark}
\label{rem:independent_initcond}
The design of the feedback GCSC can be made independent of a specific initial condition 
$x_0 \in \mathcal{B}(r)$, where $r>0$ is given. For a given $\delta>0$, the conditions 
\eqref{eq:gc_test_LMI2} and \eqref{eq:LMI_existence2} can be equivalently written as an 
LMI in the variable $P_{\alpha} \succ 0$, given by $P_{\alpha} \prec \frac{\delta}{r^2} I_n$; see 
\cite{Jiao:2019TAC}. As a result, the LMI set in \eqref{eq:YLMI} can be written equivalently as 
% $\tb{\mathcal{Y}}
% := \left \{
% Y \in \mathbb{R}^{n\times n}
% ~\big|~
% Y \succ 0,\;
% Y - \tfrac{r^2}{\delta} I_n \succ 0,\;
% \tb{\mathcal{N}}_{B}^{\prime}\Omega(Y)\tb{\mathcal{N}}_{B} \prec 0
% \right \}.$
\begin{align}
    \tb{\mathcal{Y}_{\alpha}} := \big \{ Y \in \mathbb{R}^{n\times n} ~\big|~ & Y \succ 0,\; Y - \tfrac{r^2}{\delta} I_n \succ 0, \notag \\
    & {\mathcal{N}_{\hat{B}} }^\prime \Omega_{\alpha}(Y)\mathcal{N}_{\hat{B}} \prec 0 \big \}.
\end{align}
\end{remark} 
\begin{remark}
\tb{Unlike problem $\mathrm{PO}$ in \eqref{eq:PO}, the $\mathrm{GC}$ problem in \eqref{eq:GCproblem} has an additional design parameter $\delta>0$, which plays a crucial role in the existence of $F_{\mathrm{s}} \in \hat{\mathfrak{g}}_{\alpha}(\delta)$.
By Lemma~\ref{lem:monotonicity_properties}, the 
feasible set $\hat{\mathfrak{g}}_{\alpha}(\delta)$ 
varies monotonically with $\delta$, so one can 
gradually increase $\delta$ until the set becomes 
non-empty. This allows us to identify the smallest 
performance bound for which 
$F_{\mathrm{s}} \in \hat{\mathfrak{g}}_{\alpha}(\delta)$.} 
\end{remark}
\begin{remark} \tb{Distributed control of MAS, where agents 
communicate over a directed graph, can be seen as a specific case of the proposed GCSC framework. Let $n_i$ 
denote the number of state variables of agent 
$i$, so that $n = \sum_i n_i$ with dynamics 
\eqref{eq:state_dynamics}. The structured 
feedback matrix $F_{\mathrm{s}} \in 
\mathbb{R}^{m \times n}$ can be partitioned 
into blocks $F_{\mathrm{s}}(i,j) \in 
\mathbb{R}^{m_i \times n_j}$, where 
$F_{\mathrm{s}}(i,j) = 0_{m_i \times n_j}$ 
whenever $j \notin \mathsf{N}_i \cup \{i\}$ 
for all $i, j \in \mathsf{N}$. This ensures 
that agent $i$ uses information only from its 
in-neighbors (including itself), and this 
sparsity constraint induced by the network is 
equivalent to structural constraint in \eqref{eq:Feedback_cop_structured}.}  
\end{remark}
\begin{remark}
\tb{For the full state feedback case 
$C_i = I_n$ for all $i \in \mathsf{N}$, all results in this work can be obtained as a special case. In particular, the synthesis in 
Theorem~\ref{thm:GC_cop_synthesis} does not require the structural constraint 
\eqref{eq:Feedback_cop_structured}, since \eqref{eq:Feedback_cop_structured} is trivially satisfied. } 
\end{remark}
\section{Numerical illustrations} 
\label{sec:numerical}
 In this section, we illustrate the performance of the proposed feedback GCSC through numerical examples. 
\begin{example}
\tb{Recall that the existence of a Pareto 
optimal solution is unclear for the game in 
Example~\ref{exp:non_existence}. The proposed 
GCSC framework, however, yields a feasible 
structured controller for the same game 
parameters, as demonstrated below.} Let $x_0 = (1, 1.2)$. We first consider the 
non-cooperative game setting, where each agent minimizes its own cost independently. The approximate output feedback Nash equilibrium strategies are $\tilde{F}_1^{\mathrm{NE}} = -0.7593$ 
and $\tilde{F}_2^{\mathrm{NE}} = -0.4117$; see 
\cite{Engwerda:08}. Denoting 
$\tilde{F} := (\tilde{F}_1^{\mathrm{NE}}, 
\tilde{F}_2^{\mathrm{NE}})$, the corresponding 
individual costs are $\tilde{J}_1^{\mathrm{NE}}
(\tilde{F}) = 1.3939$ and 
$\tilde{J}_2^{\mathrm{NE}}(\tilde{F}) = 1.2339$. We now consider the cooperative game setting. 
For $\delta = 1.75$, the Nash bargaining 
solution (see \cite[Ch.~6]{Engwerda:05}) gives 
$\alpha^{\star} = (0.9048, 0.0952) \in 
\mathcal{A}(\delta)$. \tb{Following 
Remark~\ref{rem:computation_GCSC}, we obtain 
the feedback GCSC strategy
% \begin{align*}
%     F_{\mathrm{s}} = 
%     \begin{bmatrix}
%         -0.9818 & 0 \\ 0 & -0.6643
%     \end{bmatrix}.
% \end{align*}
$F_{\mathrm{s}} = 
   [ \begin{smallmatrix}
        -0.9818 & 0 \\ 0 & -0.6643
    \end{smallmatrix} ].$
}\tb{The individual costs under the GCSC 
strategy are $J_1(F_{\mathrm{s}}) = 1.3816$ 
and $J_2(F_{\mathrm{s}}) = 1.2125$. Since $J_i(F_{\mathrm{s}}) < 
\tilde{J}_i^{\mathrm{NE}}(\tilde{F})$ for 
$i = 1, 2$, each player achieves a lower cost 
in the cooperative setting than in the non-cooperative setting.} Moreover, 
$J_{\alpha^{\star}}(F_{\mathrm{s}}) = 1.3655 
< \delta$, verifying 
Theorem~\ref{thm:GC_verification}. \tb{Finally, 
$J_{\mathrm{OPT}} = 2.5670$ and 
$\eta_2(F_{\mathrm{s}}) = 1.0106$, giving a 
performance loss of $ \frac{J_{\mathrm{GC}}(F_{\mathrm{s}}) - 
    J_{\mathrm{OPT}}}{J_{\mathrm{OPT}}} 
    \times 100\% = 
    (\eta_2(F_{\mathrm{s}}) - 1) \times 100\% 
    = 1.06\%$
% \begin{align*}
%     \frac{J_{\mathrm{GC}}(F_{\mathrm{s}}) - 
%     J_{\mathrm{OPT}}}{J_{\mathrm{OPT}}} 
%     \times 100\% = 
%     (\eta_2(F_{\mathrm{s}}) - 1) \times 100\% 
%     = 1.06\%
% \end{align*}
relative to $J_{\mathrm{OPT}}$.}
\end{example}

\begin{example} 
\label{ex:microgrid}
We consider the microgrid tracking 
synchronization problem from 
\cite[Sec.~V]{Lewis:13}; see also \cite{Roy:22, Cappello:21}. 
Here, we compare the proposed GCSC method 
against distributed cooperative control in 
\cite{Lewis:13}. The microgrid consists of 
$4$ networked distributed generators, 
$\mathsf{N} = \{1,2,3,4\}$, with 
double-integrator dynamics given by
\begin{align}
\dot{\xi}_i = A_g \xi_i + B_g v_i,~ 
i \in \mathsf{N},
\end{align}
where 
$A_g = \left[\begin{smallmatrix} 0 & 1 \\ 0 & 0 
\end{smallmatrix}\right]$ and 
$B_g = \left[\begin{smallmatrix} 0 \\ 1 
\end{smallmatrix}\right]$. The state 
$\xi_i := [\xi_i^1~\xi_i^2]^\prime \in 
\mathbb{R}^2$ represents the terminal voltage 
and its rate of change, and $v_i \in \mathbb{R}$ 
is the primary control input. A reference 
generator evolves as $\dot{\xi}_0 = \bar{A}\xi_0$ 
with constant voltage reference 
$\xi_0 = [\xi_r~0]^\prime$. Only agent~$1$ 
accesses $\xi_0$ directly; the remaining agents 
obtain information through the directed network; 
see \cite[Fig.~7]{Lewis:13}. The objective is to synchronize the terminal voltage of each distributed generator to the fixed reference voltage, i.e.,
$\lim_{t\to\infty} \|\xi_i(t)-\xi_0 \|=0$ for all $i \in \mathsf{N}$. In \cite{Lewis:13}, this problem is addressed 
using distributed cooperative control with 
quadratic performance indices for each 
distributed generator
\begin{align}
    \bar{J}_i = \tfrac{1}{2} \int_0^\infty 
    \left((\xi_i - \xi_0)^\prime Q 
    (\xi_i - \xi_0) + v_i^\prime R v_i 
    \right) dt,~i\in \mathsf{N},
\end{align}
where $Q = \left[\begin{smallmatrix} 50000 & 0 
\\ 0 & 1 \end{smallmatrix}\right]$ and 
$R = 0.01$.
The resulting distributed controllers take the form $v_1=-cK(\xi_1-\xi_0)$,~$v_2=-cK(\xi_2-\xi_1)$,~$v_3=-cK(\xi_3-\xi_2)$,~$v_4=-cK(\xi_4-\xi_1)$, with coupling gain $c>0$ and $K=[2236\;67.6]$. 
In this work, the same problem is reformulated within the LQ differential game framework \eqref{eq:LQDG} to compare it to the proposed GCSC strategy. We introduce relative coordinate: 
$x_1:=\xi_1-\xi_0$,~$x_2:=\xi_2-\xi_1$,~$x_3:=\xi_3-\xi_2$, $x_4:=\xi_4-\xi_1$ 
and inputs $u_i:=v_i$, the dynamics become $\dot{x}_1 = A_g x_1 + B_g u_1,~\dot{x}_2 = A_g x_2 + B_g u_2 - B_g u_1$,~$\dot{x}_3 = A_g x_3 + B_g u_3 - B_g u_2$, ~$\dot{x}_4 = A_g x_4 + B_g u_4 - B_g u_1$. The following objective of each distributed generator aligned with the network topology
\begin{align}
J_i = \tfrac{1}{2} \int_0^\infty 
(x_i^\prime Q x_i + u_i^\prime R u_i)dt,~i \in \mathsf{N}, 
\end{align}
which differs from $\bar{J}_i$, $i \in \mathsf{N}$, where every agent has access to the reference 
voltage $\xi_0$.
The local information based on communication topology
$y_1=x_1$, 
$y_2= [x_1^\prime ~ x_2^\prime]^\prime$, 
$y_3= [x_2^\prime ~ x_3^\prime]^\prime$, 
$y_4= [x_1^\prime ~ x_4^\prime]^\prime$. The system parameters in \eqref{eq:state_dynamics} and \eqref{eq:output_vector} are 
$A=I_4\otimes A_g$, $B_1=\left[\begin{smallmatrix} 1 & -1&0&-1\end{smallmatrix}\right]^\prime \otimes B_g$,
$B_2=\left[\begin{smallmatrix} 0 & 1&-1&0\end{smallmatrix}\right]^\prime \otimes B_g$,
$B_3=\left[\begin{smallmatrix} 0 & 0&1&0\end{smallmatrix}\right]^\prime \otimes B_g$,
$B_4=\left[\begin{smallmatrix} 0 & 0&0&1\end{smallmatrix}\right]^\prime \otimes B_g$, and  $C_i$ matrices are obtained using the local information vector $y_i$.
The weight matrices in the objectives $ {Q}_1=\tfrac{1}{2}Q$, ${Q}_i=\tfrac{1}{2}\left[\begin{smallmatrix}0&0\\0&1 \end{smallmatrix}\right]\otimes Q$, $i=2,3,4$, and $R_i=\tfrac{1}{2}R$, $i=1,2,3,4$. For simulations, we set $c=1$, 
$\xi_i(0)=[1\;0]^\prime$, 
$\xi_0(0)=[0.95\;0]^\prime$, $\delta = 2.5$, and $\alpha = (0.25, 0.25, 0.25, 0.25)$. The existence of a Pareto optimal solution is also unclear for this example. Following Remark~\ref{rem:computation_GCSC}, we obtain the feedback GCSC strategy as follows
\begin{align*}
F_{\mathrm{s}} =
\left[\begin{smallmatrix}
-2320.7 & -84.0 & 0 & 0 & 0 & 0 & 0 & 0 \\
-1194.4 & -70.8 & -2440.5 & -88.0 & 0 & 0 & 0 & 0 \\
0 & 0 & 75.3 & -22.0 & -2350.9 & -69.4 & 0 & 0 \\
-1014.8 & -56.6 & 0 & 0 & 0 & 0 & -2205.0 & -78.2
\end{smallmatrix} \right]. 
\end{align*}
We obtain $J_{\alpha}(F_{\mathrm{s}}) = 
1.2972 < \delta$, verifying 
Theorem~\ref{thm:GC_verification}. Moreover, 
$J_{\mathrm{OPT}} = 4.9275$ and 
$J_{\mathrm{GC}}(F_{\mathrm{s}}) = 5.1887$. 
\tb{Hence, $\eta_2(F_{\mathrm{s}}) = 1.0530
< \frac{N\delta}{J_{\mathrm{OPT}}} = 2.0294$, 
verifying Corollary~\ref{cor:eta2_upperbound}, 
and corresponds to a performance loss of 
$(\eta_2(F_{\mathrm{s}}) - 1) \times 100\% = 
5.30\%$ relative to $J_{\mathrm{OPT}}$.} To compare the team cost using $\bar{J}_i$, we obtain $\sum_i \bar{J}_i = 9.7500$ under the distributed cooperative control scheme in \cite{Lewis:13}, whereas using the obtained feedback $F_{\mathrm{s}}$ gives $\sum_i \bar{J}_i = 8.2992$. Thus, the proposed feedback GCSC reduces the total team cost compared with the distributed cooperative approach in \cite{Lewis:13}.
\begin{figure}[h]
\centering
\hspace*{-0.6cm}
\begin{subfigure}[t]{0.22\textwidth}
    \centering
    \begin{tikzpicture}[scale=.50,>=latex']
	\tikzset{every pin/.append style={font=\large}}
	\begin{axis}[xmin=0,xmax=0.4,ymin=0.92, ymax=1, 
    xlabel = {$t$ (time units)},
    ylabel={$\xi_{i}^1[V]$},legend pos = north east,
		grid=both, legend style={nodes={scale=1.35 }}, 
		grid style={line width=.1pt, draw=gray!25},
        ytick={0.92,0.94,0.95,0.96,0.98,1} ] 
		\addplot[color=red!85, line width=1.8pt, solid] table{Data/GCCstate11.dat}; \addlegendentry{Agent 1: $\xi_{1}^1$};
		\addplot[color=green!85!black, line width=1.8pt, solid] table{Data/GCCstate21.dat}; \addlegendentry{Agent 2: $\xi_{2}^1$};
		\addplot[color=blue!85, line width=1.8pt, solid] table{Data/GCCstate31.dat}; \addlegendentry{Agent 3: $\xi_{3}^1$};
		\addplot[color=black!85, line width=1.8pt, solid] table{Data/GCCstate41.dat}; \addlegendentry{Agent 4: $\xi_{4}^1$}
        \addplot[purple,dashed, line width=1.4pt] coordinates {(0,0.95) (0.4,0.95)};
	\end{axis}  
\end{tikzpicture} 
    \label{fig:statesplot1GCC}
\end{subfigure}
%\hspace{0.01\textwidth}
~
\begin{subfigure}[t]{0.22\textwidth}
    \centering
    \begin{tikzpicture}[scale=.50,>=latex']
	\tikzset{every pin/.append style={font=\large}} 
	\begin{axis}[xmin=0,xmax=0.4,ymin=-1.3, ymax=0.4, 
    xlabel = {$t$ (time units)},
    ylabel={$\xi_{i}^2[V]$},legend pos = south east,
	grid=both, legend style={nodes={scale=1.35 }}, 
	grid style={line width=.1pt, draw=gray!25},] 
	\addplot[color=red!85, line width=1.8pt, solid] table{Data/GCCstate12.dat}; \addlegendentry{Agent 1: $\xi_{1}^2$};
	\addplot[color=green!85!black, line width=1.8pt, solid] table{Data/GCCstate22.dat}; \addlegendentry{Agent 2: $\xi_{2}^2$};
	\addplot[color=blue!85, line width=1.8pt, solid] table{Data/GCCstate32.dat}; \addlegendentry{Agent 3: $\xi_{3}^2$};
	\addplot[color=black!85, line width=1.8pt, solid] table{Data/GCCstate42.dat}; \addlegendentry{Agent 4: $\xi_{4}^2$}	
    \addplot[purple,dashed, line width=1.4pt] coordinates {(0,0) (0.4,0)};
\end{axis}
\end{tikzpicture} 
    \label{fig:statesplot2GCC}
\end{subfigure} \\ [-0.25ex]
\hspace*{-0.6cm}
\begin{subfigure}[t]{0.22\textwidth}
    \centering
    \begin{tikzpicture}[scale=.50,>=latex']
	\tikzset{every pin/.append style={font=\large}}
	\begin{axis}[xmin=0,xmax=0.4,ymin=0.92, ymax=1, 
    xlabel = {$t$ (time units)},
    ylabel={$\xi_{i}^1[V]$},legend pos = north east,
		grid=both, legend style={nodes={scale=1.35 }}, 
		grid style={line width=.1pt, draw=gray!25}, ytick={0.92,0.94,0.95,0.96,0.98,1} ] 
		\addplot[color=red!85, line width=1.8pt, dash dot] table{Data/lewisy11.dat}; \addlegendentry{Agent 1: $\xi_{1}^1$};
		\addplot[color=green!85!black, line width=1.8pt, dash dot] table{Data/lewisy21.dat}; \addlegendentry{Agent 2: $\xi_{2}^1$};
		\addplot[color=blue!85, line width=1.8pt, dash dot] table{Data/lewisy31.dat}; \addlegendentry{Agent 3: $\xi_{3}^1$};
		\addplot[color=black!85, line width=1.8pt, dash dot] table{Data/lewisy41.dat}; \addlegendentry{Agent 4: $\xi_{4}^1$}
          \addplot[purple,dashed, line width=1.4pt] coordinates {(0,0.95) (0.4,0.95)};
	\end{axis}  
\end{tikzpicture} 
    \label{fig:statesplot1lewis}
\end{subfigure}
% \hspace{0.01\textwidth}
~
\begin{subfigure}[t]{0.22\textwidth}
    \centering
    \begin{tikzpicture}[scale=.50,>=latex']
	\tikzset{every pin/.append style={font=\large}} 
	\begin{axis}[xmin=0,xmax=0.4,ymin=-1.3, ymax=0.4, xlabel = {$t$ (time units)},ylabel={$\xi_{i}^2[V]$},legend pos = south east,
	grid=both, legend style={nodes={scale=1.35 }}, 
	grid style={line width=.1pt, draw=gray!25},] 
	\addplot[color=red!85, line width=1.8pt, dash dot] table{Data/lewisy12.dat}; \addlegendentry{Agent 1: $\xi_{1}^2$};
	\addplot[color=green!85!black, line width=1.8pt, dash dot] table{Data/lewisy22.dat}; \addlegendentry{Agent 2: $\xi_{2}^2$};
	\addplot[color=blue!85, line width=1.8pt, dash dot] table{Data/lewisy32.dat}; \addlegendentry{Agent 3: $\xi_{3}^2$};
	\addplot[color=black!85, line width=1.8pt, dash dot] table{Data/lewisy42.dat}; \addlegendentry{Agent 4: $\xi_{4}^2$}
    \addplot[purple,dashed, line width=1.4pt] coordinates {(0,0) (0.4,0)};
\end{axis}
\end{tikzpicture} 
    \label{fig:statesplot2lewis}
\end{subfigure}
%\vspace{-6pt}
\caption{\tb{State trajectories of the terminal voltages $(\xi_{i}^1,\xi_{i}^2)$ for Example~\ref{ex:microgrid} under feedback GCSC (top row, solid lines) and distributed cooperative control \cite{Lewis:13} (bottom row, dash-dotted lines). The left panels show 
 convergence of $\xi_i^1$ to the reference 
value $0.95$, and the right panels show convergence of $\xi_i^2$ to zero.}} 
\label{fig:traj} 
\end{figure}
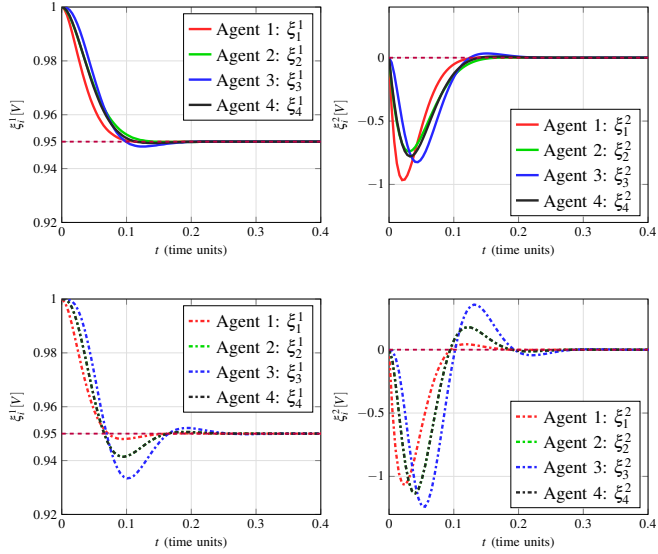
Fig. \ref{fig:traj} shows that the obtained feedback GCSC strategy improves transient response. 
\end{example}
\begin{example} 
\label{ex:5agentoutput}
We consider a set of heterogeneous agents denoted by $\mathsf{N} = \{1,2,3,4,5\}$. The agents communicate over a directed network shown in Fig.~\ref{fig:5agentOutput}.
 \begin{figure}[h] 
	\centering
	\begin{tikzpicture}[scale=.60,>=latex', inner sep=1mm, font=\small]
	\tikzstyle{single node}=[circle,auto=center,draw,minimum size=6pt,inner sep=2,fill=mygreen!20 ]
	\tikzstyle{double node}=[circle,auto=center,draw,minimum size=6pt,inner sep=2,fill=blue!20 ]
	\tikzstyle{diredge} = [draw, black!65, ->, line width=0.25mm]
	\tikzstyle{edge} =    [draw, black!65, -, line width=0.25mm]
	\def\x{2}
	%%% NODES
	\node (n1)[single node] at (0,0) {\scriptsize{$1$}};
	\node (n2)[double node] at (-.5*\x,-\x) {\scriptsize{$2$}};
	\node (n3)[single node] at (.5*\x,-\x) {\scriptsize{$3$}};
	\node (n4)[single node] at (2*\x,-\x) {\scriptsize{$4$}};
	\node (n5)[double node] at (2*\x,0) {\scriptsize{$5$}};
	
	%%% EDGES
	
	\path[diredge] (n1) --  (n2);
	\path[diredge] (n2) --  (n3);
	\path[diredge] (n3) --  (n1);
	\path[diredge] (n1) --  (n5);
	\path[diredge] (n1) --  (n4);
	\path[diredge] (n3) --  (n4);
	\path[diredge] (n4) --  (n5);
	
	%%%
\draw[fill=blue!20](6,-1) circle[radius=0.25];
\draw[fill=mygreen!20](6,-1.7) circle[radius=0.25];
% \draw(6,-1)node[right,xshift=5]{Agent state $x_i \in \mathbb{R}^2$};
% \draw(6,-1.7)node[right,xshift=5]{Agent state $x_i \in \mathbb{R}$};
\draw(6,-1)node[right,xshift=5]{Second-order agent};
\draw(6,-1.7)node[right,xshift=5]{First-order agent};
\end{tikzpicture}
    %\vspace{-1pt}
	\caption{\tb{Directed communication graph of a 5-agent heterogeneous MAS considered in Example~\ref{ex:5agentoutput}.}}
	\label{fig:5agentOutput}
\end{figure}
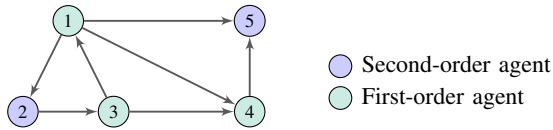
The LTI dynamics of each agent $i \in \mathsf{N}$ is given by: 
\begin{align}
    \dot{x}_i(t)= A_{ii} x_i(t)+ B_{ii} u_i(t) + \sum_{j \in \mathsf{N}_i} B_{ij} u_j(t), 
\label{eq:MAS_dynamics}
\end{align}
where $x_i \in \mathbb{R}, ~i=1,3,4$, $x_i \in \mathbb{R}^{2},~i=2,5$, and $u_i \in \mathbb R,~i = 1,3,4,5$, $u_i = \mathbb{R}^2,~i = 2$. Next, we denote $x_i := p_i \in \mathbb{R}$ for $i=1,3,4$, and $x_i := [p_i \; v_i]^{\prime} \in \mathbb{R}^2$ for $i=2,5$. 
We consider the system parameters from \cite{Cappello:21} and given by $A_{11}=0$, $A_{22}=\left[\begin{smallmatrix} 1 & 1 \\ 1 & 1 \end{smallmatrix}\right]$, $A_{33}=1$, $A_{44}=2$, and $A_{55}=\left[\begin{smallmatrix} 0 & 1 \\ 0 & 0 \end{smallmatrix}\right]$. The input matrices $B_{ii}$ and $B_{ij}, i \neq j$ matrices are $ B_{11}=1$, $ B_{21} = 
    \left[ \begin{smallmatrix} 0.3 \\ 0.2 
        \end{smallmatrix} \right]$, $B_{41}=0.1$, $B_{51}= 
    \left[ \begin{smallmatrix}
      0.2 \\  0
     \end{smallmatrix} \right]$, $B_{22} = 
     \left[\begin{smallmatrix}
      1 &0 \\
      0 &-1
     \end{smallmatrix} \right]$, $B_{32}= 
     \left[\begin{smallmatrix}
      0 &0.2
     \end{smallmatrix} \right]$, $B_{33}=1$, $ B_{13}=0.2$, $B_{43}=0.3$, $B_{44}=1$, $B_{54}=
     \left[\begin{smallmatrix}
      0.1 \\
      0
     \end{smallmatrix} \right]$, $B_{55} = 
    \left[ \begin{smallmatrix}
      0 \\
      2
     \end{smallmatrix} \right]$. Using the individual agent dynamics \eqref{eq:MAS_dynamics}, the overall system dynamics can be expressed in the form of \eqref{eq:state_dynamics}. From Fig.~\ref{fig:5agentOutput}, the in-neighbor set of each agent set is $\mathsf{N}_1 = \{3\}, \mathsf{N}_2 = \{1\}, \mathsf{N}_3 = \{2\}, \mathsf{N}_4 = \{1,3\}$, and $\mathsf{N}_5 = \{1,4\}$. The objective of the agent $i \in\{2,5\}$ is given by
     \begin{align}
        J_i = \int_{0}^{\infty} \Big(\sum_{j \in \mathsf N_i } (p_i -p_j)^2 + \sum_{ j \in \mathsf N_i \cap \{2,5\}}(v_i-v_j)^2 +  u_i^\prime R_i u_i \Big)dt, 
     \end{align}
    and the objective of agent $i\in \{1,3,4\}$ is given by
     \begin{align}
         J_i=\int_0^\infty \Big(\sum_{j \in \mathsf N_i } (p_i -p_j)^2+u_i^\prime R_i u_i \Big)dt. 
     \end{align}
We obtain the local information vectors $y_1=[x_1^\prime ~ x_3^\prime]^\prime$, $y_2= [x_1^\prime ~ x_2^\prime]^\prime$, 
$y_3= [x_2^\prime ~ x_3^\prime]^\prime$, 
$y_4= [x_1^\prime ~ x_3^\prime ~ x_4^\prime]^\prime$, $y_5= [x_1^\prime ~ x_4^\prime ~ x_5^\prime]^\prime$, $C_i$, and $Q_i$ from the network structure. 
We consider $R_i=1,~i=1,3,4,5,~R_i=I_2,~i=2$. 
 For simulation purpose, we choose, $x_0= (-0.3,-0.5,-0.4,-0.2,-0.1,-0.3,-0.4)$, and $\alpha_1 = \cdots = \alpha_5 = 0.2$, and $\delta = 0.25$. We obtain the following feedback GCSC strategy using the procedure described in Remark \ref{rem:computation_GCSC}. 
\begin{align*}
F_{\mathrm{s}}= 
\left[\begin{smallmatrix}
-1.3392 & 0        & 0        & 0.5544  & 0        & 0        & 0 \\
 0.9748 & -3.9132  & -2.4520  & 0       & 0        & 0        & 0 \\
-0.4568 & 2.4189   & 2.8231   & 0       & 0        & 0        & 0 \\ 
 0      & 0.2554   & -0.3968  & -3.9270 & 0        & 0        & 0 \\ 
 0.8975 & 0        & 0        & 0.1809  & -5.5484  & 0        & 0 \\ 
 0.7665 & 0        & 0        & 0       & -0.3011  & -1.4171  & -1.8687
 \end{smallmatrix} \right]
\end{align*}
For this problem, the existence of a Pareto optimal solution is unclear. We obtain $J_{\mathbf{\alpha}}(F_{\mathrm{s}}) = 0.2154 < \delta$, which verifies Theorem~\ref{thm:GC_verification}. Furthermore, we compute $J_{\mathrm{OPT}} = 1.0269$, and using the obtained $F_{\mathrm{s}}$, we obtain $J_{\mathrm{GC}}(F_{\mathrm{s}}) = 1.0768$. \tb{Hence, $\eta_2(F_{\mathrm{s}}) = 1.0486 
< \frac{N\delta}{J_{\mathrm{OPT}}} = 1.2173$, 
verifying Corollary~\ref{cor:eta2_upperbound}, 
and corresponds to a performance loss of 
$(\eta_2(F_{\mathrm{s}}) - 1) \times 100\% = 
4.86\%$ relative to $J_{\mathrm{OPT}}$.} Fig. \ref{fig:trajhetero} shows that the agents' trajectories converge asymptotically to zero.
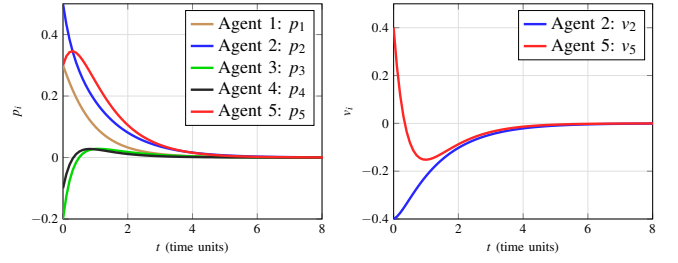
\begin{figure}[t]   \centering
\hspace*{-0.8cm}
	\begin{subfigure}[t]{.20\textwidth} 
		\begin{tikzpicture}[scale=.50,>=latex']
	\tikzset{every pin/.append style={font=\large}}
	\begin{axis}[xmin=0,xmax=8,ymin=-0.2, ymax=0.5, xlabel = {$t$ (time units)},ylabel={$p_{i}$},legend pos = north east,
		grid=both, legend style={nodes={scale=1.35 }}, 
		grid style={line width=.1pt, draw=gray!25},] 
		\addplot[color=brown!85, line width=1.8pt] table{Data/position1.dat}; \addlegendentry{Agent $1$: $p_{1}$};
		\addplot[color=blue!85, line width=1.8pt] table{Data/position2.dat}; \addlegendentry{Agent $2$: $p_{2}$};
		\addplot[color=green!85!black, line width=1.8pt] table{Data/position3.dat}; \addlegendentry{Agent $3$: $p_{3}$};
		\addplot[color=black!85, line width=1.8pt] table{Data/position4.dat}; \addlegendentry{Agent $4$: $p_{4}$}
		\addplot[color=red!85, line width=1.8pt] table{Data/position5.dat}; \addlegendentry{Agent $5$: $p_{5}$}
	\end{axis}  
\end{tikzpicture}  
		\label{fig:position1plotGCC}
	\end{subfigure}	   \qquad  
	\begin{subfigure}[t]{.20\textwidth} 
		\begin{tikzpicture}[scale=.50,>=latex']
	\tikzset{every pin/.append style={font=\large}}
	\begin{axis}[xmin=0,xmax=8,ymin=-0.4, ymax=0.5, xlabel = {$t$ (time units)},ylabel={$v_{i}$},legend pos = north east,
		grid=both, legend style={nodes={scale=1.35}}, 
		grid style={line width=.1pt, draw=gray!25},] 
		\addplot[color=blue!85, line width=1.8pt] table{Data/velocity2.dat}; \addlegendentry{Agent $2$: $v_{2}$};
		\addplot[color=red!85, line width=1.8pt] table{Data/velocity5.dat}; \addlegendentry{Agent $5$: $v_{5}$};
	\end{axis}  
\end{tikzpicture}  
		\label{fig:velocity2plotGCC}
	\end{subfigure} 
    %\vspace{-6pt}
	\caption{\tb{State trajectories of the agents for Example~\ref{ex:5agentoutput} under the proposed GCSC strategy. 
The left panel shows the position trajectories, while the right panel shows the velocity trajectories. 
Both asymptotically converge to zero.}} 
	\label{fig:trajhetero} 
\end{figure}
\end{example}

\section{Conclusion}
\label{sec:concl}
In this work, we introduced the notion of 
feedback GCSC for infinite-horizon LQ-CDGs under output feedback, which provides a tractable approach when computing Pareto optimal controls under output feedback is difficult. At a feedback GCSC, the weighted team cost is upper-bounded by a prescribed threshold while satisfying the structural constraint. We established monotonicity properties of the GCSC set and the admissible weight set. We showed that Pareto optimal controls, whenever they exist, belong to the class of feedback GCSCs. We also provided 
performance measures for the Pareto optimal 
controls and the proposed GCSC relative to the 
output feedback optimal control. Verification 
and synthesis conditions were derived using LMIs, with the synthesis formulation being convex and requiring no SDP relaxation. Future directions include extending the GCSC concept to scalar CDGs, dynamic output feedback, and stochastic CDGs.
\vspace{0.1em}

\section*{Acknowledgment}
The authors thank Dr. P. V. Reddy, IIT Madras, for his valuable feedback on the manuscript.
% \noindent\textbf{Acknowledgment:} The authors thank Dr. P. V. Reddy, IIT Madras, for his valuable feedback on the manuscript.
% \vspace{-1.1em}
\bibliographystyle{IEEEtran}
\bibliography{Pareto}

%%%%%%%%%%%%%%%%%%%%%%%%%%%%%%%%
% \section*{Notes}
% \begin{enumerate}[label=(N\arabic*)] % Customize list appearance
%   \forlistloop{\item}{\noteslist}%
% \end{enumerate}

\end{document}